\documentclass[11pt]{article}
\usepackage{amsmath}
\usepackage{amssymb}
\usepackage{amsthm}
\usepackage[mathscr]{eucal}
\usepackage{graphics}
\usepackage{endnotes}

\usepackage{epsfig}

\newtheorem{theorem}{Theorem}[section]
\newtheorem{lemma}{Lemma}[section]

\newtheorem{corollary}{Corollary}[section]

\newtheorem{acknowledgement}{Acknowledgement}
\newtheorem{remark}{Remark}

\input{epsf}

\title{Evolution algebras generated by Gibbs measures}

\author{Utkir A. Rozikov\textsuperscript{1},\ \
Jianjun Paul Tian\textsuperscript{2}  \footnote{Corresponding author} \\
{\small \textsuperscript{1} Institute of Mathematics and
Information Technologies,\hfill{\ }}\\
\ \ {\small 29, F. Hodjaev str., Tashkent, 100125, Uzbekistan \hfill
{\
}}\\
{\small \textsuperscript{2} Department of Mathematics, College of William and Mary,\hfill{\ }}\\
\ \ {\small Williamsburg, Virginia, 23187-8795, USA\hfill {\
}}\\
{\small Email: rozikovu@yandex.ru, jptian@math.wm.edu \hfill {\ }}}

\date{}

\begin{document}

\maketitle

\begin{abstract}
In this article we study algebraic structures of function spaces
defined by graphs and state spaces equipped with Gibbs measures by
associating evolution algebras. We give a constructive description
of associating evolution algebras to the function spaces (cell
spaces) defined by graphs and state spaces and Gibbs measure $\mu$.
For finite graphs we find some evolution subalgebras and other
useful properties of the algebras. We obtain a structure theorem for
evolution algebras when graphs are finite and connected. We prove
that for a fixed finite graph, the function spaces has a unique
algebraic structure since all evolution algebras are isomorphic to
each other for whichever Gibbs measures assigned. When graphs are
infinite graphs then our construction allows a natural introduction
of thermodynamics in studying of several systems of biology, physics
and mathematics by theory of evolution algebras.

\end{abstract}

\section{Introduction}

The concept of evolution algebras lies between algebras and
dynamical systems. Algebraically, evolution algebras are
non-associative Banach algebra; dynamically, they represent discrete
dynamical systems. Evolution algebras have many connections with
other mathematical fields including graph theory, group theory,
stochastic processes, mathematical physics etc.

In book \cite{tian} the foundation of evolution algebra theory and
applications in non-Mendelian genetics and Markov chains are
developed, with pointers to some further research topics.

Gibbs measures are familiar subjects in various areas of applied
probability. Originally they were defined in the framework of
lattice spin systems to describe thermodynamic equilibrium states
\cite{h-o}\cite{preston}\cite{sinai}.

In the present paper we explore some algebraic structures of
function spaces defined by graphs and state spaces equipped with
Gibbs measures by associating evolution algebras. The study of these
function spaces defined by graphs and state spaces equipped with
Gibbs measures was inspired by Mendelian genetics. For finite graphs
we find some evolution subalgebras and other useful properties of
the algebras. We obtain a structure theorem for evolution algebras
when graphs are finite and connected. We also prove that for a fixed
finite graph, the function spaces (cell spaces) has a unique
algebraic structure since all evolution algebras are isomorphic each
other for whichever Gibbs measures assigned. Note that each
evolution algebra can be uniquely determined by an assignment of its
structural coefficients which can be arranged into a quadratic
matrix. For infinite graphs we must use limit Gibbs measures, then
one of the key problems is to determine the entries (structural
coefficients) of the matrices which are already infinite sizes; the
second problem is to investigate the evolution algebras which
correspond to these matrices. These investigations allows to a
natural introduction of thermodynamics in the study of the evolution
algebras. More precisely, if graph $G$ is infinite (countable) then
one can associate a limit Gibbs measure $\mu$ by a Hamiltonian $H$
(defined on $G$) and temperature $T>0$
\cite{h-o}\cite{preston}\cite{sinai}. It is known that depending on
the Hamiltonian and the values of the temperature the measure $\mu$
can be non unique. In this case there is a phase transition of the
physical system with the Hamiltonian $H$. We ask: how the
thermodynamics (the phase transition) will affect properties of
evolution algebras corresponding to Gibbs measures of the
Hamiltonian $H$, and how evolution algebras reflect thermodynamics.
We will give some comments about this problem.

This paper is organized as follows. In Section 2, we give some
preliminaries about graphs, evolution algebras, and Gibbs measures.
In Section 3, for finite graphs and finite state spaces, evolution
algebras are constructed. When graphs are finite and connected, the
structure theorem is obtained. We also study how Gibbs measures
affect evolution algebras. In Section 4, we present two simple
examples to illustrate the theorems in Section 3. In Section 5, we
constructively define evolution algebras when graphs are countable.
We also give an example to show that the limit evolution algebra for
a given graph and a series of Gibbs measures is not unique. We list
several open problems at the end.

\section{Preliminaries}

\noindent {\bf Graphs:} A graph $G$ is an ordered pair of disjoint
sets $(\Lambda, L)$ such that $L$ is a subset of the set
$\Lambda^{(2)}$ of unordered pairs of $\Lambda$. The set $\Lambda$
is the set of vertices and $L$ is the set of edges. An edge
$\{x,y\}$ is said to join the vertices $x$ and $y$. If $\{x,y\}\in
L$, then $x$ and $y$ are called neighboring vertices of $G$. We say
that $G'=(\Lambda', L')$ is a subgraph of $G=(\Lambda, L)$ if
$\Lambda'\subset \Lambda$ and $L'\subset L$. A graph is connected if
for every pair $x, y$ of distinct vertices there is a path from $x$
to $y$. A maximal connected subgraph is a component of the graph.

\noindent {\bf Evolution algebras:} Let $(A,\cdot)$ be an algebra
over a field $K$. If it admits a basis $e_1,e_2,...$, such that
\[e_i\cdot e_j=0, \ \ \mbox{if}\ \ i\ne j,\]
\[e_i\cdot e_i=\sum_{k}a_{ik}e_k, \ \ \mbox{for any}\ \ i,\] we then
call this algebra an evolution algebra.

Evolution algebras have following elementary properties (see
\cite{tian})

--Evolution algebras are not associative, in general.

--Evolution algebras are commutative, flexible.

--Evolution algebras are not power-associative, in general.

--Direct sums of evolution algebras are also evolution algebras.

--Kronecker products of evolutions algebras are also evolution
algebras.

\noindent {\bf Function spaces (cell spaces) and Hamiltonian:}  Let
$G=(\Lambda, L)$ be a graph, and $\Phi$ be a state space with finite
many states. For any $A\subseteq \Lambda$, a cell $\sigma_{A}$ on
$A$ is defined as a function
 $x\in A\to\sigma_{A}(x)\in\Phi=\{1,2,...,q\}$; the set of all cells
 coincides with $\Omega_A=\Phi^{A}$. We denote $\Omega=\Omega_\Lambda$ and
$\sigma=\sigma_{\Lambda}.$

The energy of the cell $\sigma\in \Omega$ is given by the formal
Hamiltonian
\[H(\sigma)=\sum\limits_{A\subset \Lambda}I(\sigma_{A}) \eqno (1)\]
where $I(\sigma_{A}): \Omega_A\to R$ is a given potential function.

For a finite domain $D\subset \Lambda$ with the boundary condition
$\varphi_{D^c}$ given on its complement $D^c=\Lambda\setminus D,$
the conditional Hamiltonian is \[H(\sigma_{D}\big |
\varphi_{D^c})=\sum_{A\subset \Lambda: A\cap D\ne
\emptyset}I(\sigma_{A}), \eqno(2)\] where
\[\sigma_{A}(x)=\left\{\begin{array}{ll}
\ \sigma(x) & \textrm{if \ \
$ x\in A\cap D$}\\
\varphi(x)& \textrm{if \ \ $x\in A\cap D^c.$}\\
\end{array}\right.\]

\noindent {\bf Gibbs measures:} We consider a standard sigma-algebra
${\cal B}$ of subsets of $\Omega$ generated by cylinder subsets; all
probability measures are considered on $(\Omega,{\cal B})$. A
probability measure $\mu$ is called a Gibbs measure (with
Hamiltonian $H$) if it satisfies the DLR equation: for all finite
$D\subset \Lambda$ and $\sigma_D\in\Omega_D$:
\[\mu\left(\left\{\omega\in\Omega:\;
\omega\big|_{D}=\sigma_D\right\}\right)= \int_{\Omega}\mu ({\rm
d}\varphi)\nu^{D}_\varphi (\sigma_D),\eqno (3)\] where
$\nu^{D}_{\varphi}$ is the conditional probability: \[
\nu^{D}_{\varphi}(\sigma_D)=\frac{1}{Z_{D,\varphi}}\exp\;\left(-\beta
H \left(\sigma_D\big |\,\varphi_{D^c}\right)\right). \eqno (4)\]
Here $\beta={1\over T}, T>0$ temperature and $Z_{D, \varphi}$ stands
for the partition function in $D$, with the boundary condition
$\varphi$: \[Z_{D, \varphi}=
\sum_{{\widetilde\sigma}_D\in\Omega_{D}} \exp\;\left(-\beta H
\left({\widetilde\sigma}_D\,\big |\,\varphi_{D^c}\right)\right).\]

\section{Construction of Evolution algebras for finite graphs}

Let $G=(\Lambda, L)$ be a finite graph without loops and multiple
edges. Furthermore, let $\Phi$ be a finite set. Denote by $\Omega$
the set of all cells $\sigma:\Lambda\to\Phi$.

One interpretation of the set $\Omega^{2}=\Omega\times \Omega $
could be the set of all pairs of "parents".

Let $\{\Lambda_{i}, i=1,...,m\}$ be the set of all maximal connected
subgraphs (components) of the graph $G$.

For any $M\subset \Lambda$ and $\sigma\in \Omega$ denote
$\sigma(M)=\{\sigma(x): x\in M\}$. We say $\sigma(M)$ is a subcell
iff $M$ is a maximal connected subgraph of $G$.

Fix two cells $\varphi, \psi \in \Omega,$ (i.e. fix a pair of
parents $\theta=(\varphi, \psi)$) and set
\[\Omega_{\theta}=\{\omega\in \Omega: \omega(\Lambda_{i})=\varphi(\Lambda_{i}) \ \
{\rm or} \ \ \omega(\Lambda_{i})=\psi(\Lambda_{i}) \ \ {\rm for \ \
all} \ \ i=1,...,m\}.\eqno(5)\]

\begin{remark}
The set $\Omega_{\theta}$ can be interpreted as the set of all
possible children of the pair of parents $\theta$. A child $\omega$
can be born from $\theta$ if it only consists the subcells of its
parents $\theta$. Such a set was first considered in \cite{n-n} and
in the general form in \cite{rozikov}.
\end{remark}

Denote $\Omega^{2}_{\theta}=\{\psi=(\psi_1,\psi_2): \psi_1,\psi_2\in
\Omega_{\theta}\}$. This set can be interpreted as a set of all
possible pairs of parents which can be generated from the parents
$\theta\in \Omega{^2}$.

Let $S(\Omega)$ be the set of all probability measures defined on
the finite set $\Omega$. Now let $\mu\in S(\Omega)$ be a probability
measure defined on $\Omega$ such that $\mu(\sigma)>0$ for any cell
$\sigma\in \Omega;$ i.e $\mu$ is a Gibbs measure with some potential
\cite{preston}.

Consider measure $\mu^2=\mu\times\mu$ on $\Omega{^2}$.

The (heredity) coefficients $P_{\varphi\psi}$ on $\Omega{^2}$ are
defined as \[P_{\varphi\psi}=\left\{\begin{array}{ll}
{\mu^2(\psi)\over \mu^2(\Omega{^2}_\varphi)}, \ \ {\rm if} \ \
\psi\in
\Omega{^2}_\varphi,\\
0 \ \ {\rm otherwise}.\\
\end{array}\right.\eqno (6)\]
Obviously, $P_{\varphi\psi}\geq 0$  and $\sum_{\psi\in \Omega{^2}}
P_{\varphi\psi}=1$ for all $\varphi\in \Omega{^2}.$

Now we can define an evolution algebra
$\mathcal{E}=\mathcal{E}(G,\Phi,\mu)$ generated by generators from
$\Omega^{2}$ and following defining relations:

\[\left\{\begin{array}{ll} \varphi^2=\sum_{\psi\in
\Omega{^2}}P_{\varphi\psi}\psi,  \ \ {\rm for} \
\ \varphi\in \Omega{^2},\\
\varphi\cdot \psi=0 \ \ {\rm if} \ \ \varphi\ne\psi, \ \ \varphi,\psi\in \Omega{^2}.\\
\end{array}\right.\eqno (7)\]

To study properties of the evolution algebra $\mathcal{E}$ we shall
use properties of the set $\Omega{^2}_{\sigma}$.

\begin{lemma}\label{lema-1}
If $\tau\in \Omega{^2}_{\sigma}$ then $\Omega{^2}_{\tau}\subseteq
\Omega{^2}_{\sigma}$, where $\sigma\in\Omega^{2}.$
\end{lemma}
\begin{proof}
Let $G$ be a (finite) graph and $\{\Lambda_{i}, i=1,...,m\}$ the set
of all maximal connected subgraphs of $G$. Denote by
$\Phi_{i}=\Phi^{\Lambda_{i}}$ the set of all subcells defined on
$\Lambda_{i}$, $i=1,...,m$. Then any cell $\rho\in\Omega$ has
representation $\rho=(\rho_{1},...,\rho_{m})$, with $\rho_{i}\in
\Phi_{i}$, $i=1,...,m$.

For $\sigma=(\varphi,\psi)\in \Omega^{2}$ with
$\varphi=(\varphi_{1},...,\varphi{_m}), \
\psi=(\psi_{1},...,\psi_{m})\in \Omega$, we have
\[\Omega^{2}_{\sigma}=\left\{\tau=(\nu,\omega)=((\nu_1,...,\nu_m),(\omega_{1},...,\omega_{m}))
: \nu_i, \omega_{i}\in \{\varphi_{i},\psi_{i}\}, i=1,...m\right\}.\]
Take $\tau=(\nu, \omega)\in \Omega^{2}_{\sigma}$ then
$\nu_i,\omega_{i}\in\{\varphi_{i},\psi_{i}\}$ for any $i=1,...,m$.
Thus for arbitrary $\xi=(\xi^1, \xi^2)\in \Omega^{2}_{\tau}$ we have
$\xi^1_i, \xi^2_i\in \{\nu_i,\omega_{i}\}\subseteq
\{\varphi_{i},\psi_{i}\}$ for any $i=1,...,m$, i.e. $\xi\in
\Omega^{2}_{\sigma}$. This completes the proof.
\end{proof}

As a corollary of Lemma \ref{lema-1} we have

\begin{lemma}\label{lema-2}
If $\tau\in \Omega_{\sigma}$ and $\sigma\in \Omega^{2}_{\tau}$, then
$\Omega^{2}_{\sigma} = \Omega^{2}_{\tau}$.
\end{lemma}

Note that for any $\sigma\in \Omega^{2}$ we have $\sigma\in
\Omega^{2}_{\sigma}$. The next lemma gives condition on $\sigma$
under which $\Omega^{2}_{\sigma}$ contains only $\sigma$ i.e
$\Omega^{2}_{\sigma}=\{\sigma\}$.

\begin{lemma}\label{lema-3}
$\Omega^{2}_{\sigma}=\{\sigma\}$ if and only if
$\sigma=(\varphi,\psi)$ with $\varphi=\psi$.
\end{lemma}
\begin{proof}
By definition of  $\Omega_{\sigma}^{2}$ with $\sigma=(\varphi,\psi)$
we have $(\varphi,\varphi), (\varphi,\psi), (\psi,\varphi),
(\psi,\psi)\in \Omega^{2}_{\sigma}$. This completes the proof.
\end{proof}

\begin{theorem}\label{thm-1}
The set $\Omega^{2}_{\sigma}$ generates an evolution subalgebra of
the evolution algebra $\mathcal{E}$ for any $\sigma\in \Omega^{2}$.
\end{theorem}
\begin{proof}
Take $\varphi\in \Omega^{2}_{\sigma}$. By (6) and (7) we get
\[\varphi{^2}=\sum_{\psi\in\Omega^{2}}P_{\varphi\psi}\psi=\sum_{\psi\in\Omega^{2}_{\sigma}}P_{\varphi\psi}\psi.\]
According to Lemma \ref{lema-1} we have
$\Omega^{2}_{\varphi}\subseteq \Omega^{2}_{\sigma}$. Then if we
denote the subalgebra that is generated by $\Omega^{2}_{\sigma}$ by
$\langle \Omega^{2}_{\sigma} \rangle$, it is clear that
$\varphi{^2}\in \langle \Omega{^2}_{\sigma}\rangle$, whenever
$\varphi\in \Omega^{2}_{\sigma}$. Thus $\Omega^{2}_{\sigma}$
generates an evolution algebra. Theorem is proved.
\end{proof}

As a corollary of Theorem \ref{thm-1} and Lemma \ref{lema-3} we have
a corollary.
\begin{corollary}
The evolution algebra $\mathcal{E}$ has the evolution subalgebra
$\langle \theta\rangle$ which is generated by one generator
$\theta=(\varphi,\varphi)$ for any $\varphi\in \Omega$.
\end{corollary}

\begin{remark}
1) Note that $\langle\theta\rangle$ is an one dimensional subalgebra
and number of such algebras is equal to $|\Omega
|=|\Phi|^{|\Lambda|}$ where $|M|$ denotes number of elements
(cardinality) of $M$.

2) For any $\sigma=(\varphi,\psi)\in \Omega{^2}$ with $\varphi\ne
\psi$ to construct minimal set $\Omega^{2}_{\sigma}$ we must take a
connected graph $G$. Then $\Omega^{2}_{\sigma}$  contains four
elements only.
\end{remark}

\begin{remark}
For any $\sigma=(\varphi,\psi)\in \Omega^{2}$ with $\varphi\ne \psi$
we have $\omega_{1}=(\varphi,\varphi), \omega_{2}=(\psi,\psi)\in
\Omega^{2}_{\sigma}$. By Lemma \ref{lema-3} we get
$\Omega^{2}_{\omega_{i}}=\{\omega_{i}\}$, $i=1,2.$
\end{remark}

We put these remarks together, and have a structural theorem for
finite connected graphs.

\begin{theorem}\label{structure}
If the finite graph $G=(\Lambda, L)$ is connected and $|\Lambda|=n$,
then $\mathcal{E}=\mathcal{E}(G,\Phi,\mu)$ has dimension of
$k^{2n}$, where $k=|\Phi|$. Furthermore, $\mathcal{E}(G,\Phi,\mu)$
has $\frac{1}{2}k^{n}(k^{n}-1)$ 4-dimensional evolution subalgebras,
and $k^{n}$ 1-dimensional evolution subalgebras.
$\mathcal{E}(G,\Phi,\mu)$ has a two-level hierarchy, 0-th level has
$k^{n}$ 1-dimensional subalgebras and 1st level has
$\frac{1}{2}k^{n}(k^{n}-1)$ 2-dimensional subalgebras.
\end{theorem}

For a fixed graph $G$, when we take different Gibbs measures for the
function space $\Omega$, the algebra
$\mathcal{E}=\mathcal{E}(G,\Phi,\mu)$ will be different. What are
relations among these algebras? or How do Gibbs measures affect
algebras? The following theorem answers this question when the graph
$G$ is finite.

\begin{theorem}\label{structure-0}
For a fixed finite graph $G=(\Lambda, L)$, all evolution algebras
$\mathcal{E}(G,\Phi,\mu)$ defined by different Gibbs measures $\mu$
are isomorphic each other. When $G$ is connected, their hierarchies
have two levels, share the same skeleton.
\end{theorem}
\begin{proof}
Take two Gibbs measures $\mu_{1}$ and $\mu_{2}$, let's assign a 1-1
and onto map $\digamma$ between $\mathcal{E}(G,\Phi,\mu_{1})$ and
$\mathcal{E}(G,\Phi,\mu_{2})$. These two algebras share the same
generator set $\Omega=\{\varphi\mid \varphi: \Lambda \longrightarrow
\Phi\}$. We define $\digamma(\varphi)=\varphi$, and then linearly
extend this map to the whole algebra. Let's denote, the structural
coefficients for the algebra defined by the Gibbs measure $\mu_{i}$
by $P_{\varphi\psi}(\mu_{i})$. From the properties of Gibbs
measures, $P_{\varphi\psi}(\mu_{1})\neq 0$ if and only if
$P_{\varphi\psi}(\mu_{2})\neq 0$, or $P_{\varphi\psi}(\mu_{1})= 0$
if and only if $P_{\varphi\psi}(\mu_{2})= 0$. Then the map
$\digamma$ keeps all generating relations, and keep all algebraic
relations too. $\digamma$ is an algebraic map. Therefore,
$\mathcal{E}(G,\Phi,\mu_{1})$ is isomorphic to
$\mathcal{E}(G,\Phi,\mu_{2})$ \cite{tian}.

When $G$ is connected, from Theorem \ref{structure} we know that the
hierarchy of the algebra $\mathcal{E}(G,\Phi,\mu)$ has two levels
for any Gibbs measure $\mu$. As to these algebra's common skeleton
which is homomorphic to each of them \cite{tian}, it keeps all
algebraic relations except the dimensions. So, in the 0-th level the
skeleton has $n^{k}$ 1-dimensional subalgebras and in the 1st level
the skeleton has $\frac{1}{2}n^{k}(n^{k}-1)$ 1-dimensional
subalgebras.
\end{proof}

Actually, for these algebras $\mathcal{E}(G,\Phi,\mu)$ differed by
Gibbs measure $\mu$, when $G$ is finite and connected graph, they
represent a similar dynamical system. For example, there are two
flows directs from each subalgebra at 1st level of their hierarchy
to two different subalgebras at 0-th level of the hierarchy. We will
illustrate this point in the following Example 1.

For two elements interpreted as parents $\sigma, \tau\in \Omega^{2}$
we denote $\tau\prec \sigma$ if $P_{\sigma\tau}>0$ (see the formula
(6)). By our construction $\tau\prec\sigma$ if and only if $\tau\in
\Omega^{2}_{\sigma}$.

\begin{theorem}
For any $\sigma\in \Omega^{2} $ there exists $n\in
\{1,2,...,|\Omega^{2}_{\sigma}|\}$ and $\tau_1,\tau_2,...,\tau_n\in
\Omega^{2}_{\sigma}$ such that
$\tau_n\prec\tau_{n-1}\prec...\prec\tau_1\prec\sigma$ with
$|\Omega^{2}_{\tau_i}|>2$, $i=1,...,n-1$ and
$\Omega^{2}_{\tau_n}=\{\tau_n\}.$
\end{theorem}
\begin{proof}
If $\Omega^{2}_{\sigma}=\{\sigma\}$ then $n=1$ and $\tau_1=\sigma$.
If there exists $\tau_1\in \Omega^{2}_{\sigma}$ such that $\tau_1\ne
\sigma$ then consider $\Omega^{2}_{\tau_1}$ which is subset (by
Lemma \ref{lema-1}) of $\Omega^{2}_{\sigma}$. If there exists
$\tau_2\in \Omega^{2}_{\tau_1}$ such that $\tau_2\ne \tau_1, \sigma$
then $\Omega^{2}_{\tau_2}\subseteq \Omega^{2}_{\tau_1}$ and so on.
Iterating this argument at most $|\Omega^{2}_{\sigma}|$ time, we get
(see Remark 3) $\Omega^{2}_{\tau_n}=\{\tau_n\}$ with some $n\in
\{1,...,|\Omega^{2}_{\sigma}|\}$.
\end{proof}

\begin{corollary}
For any $\sigma\in \Omega^{2}$ there are subalgebras
$\langle\Omega^{2}_{\tau_i}\rangle$, $i=1,...,n$ such that
\[\langle\Omega^{2}_{\tau_n}\rangle \subseteq
\langle\Omega^{2}_{\tau_{n-1}}\rangle \subseteq... \subseteq
\langle\Omega^{2}_{\sigma}\rangle.\]
\end{corollary}

Recall the hierarchy structure theorem of evolution algebras that,
for any element of the algebra, there is a sequence of subalgebras
which ends at a simple subalgebra such that each of them contains
this element. For the algebras $\mathcal{E}(G,\Phi,\mu)$, this
corollary tells how the flows direct or where their children will
be.

\section{Examples}

\noindent {\bf Example 1:} Consider $\Lambda=\{1,2\},\ L=\{<1,2>\}$,
i.e. $G$ is connected graph with one edge $<1,2>$. Take $\Phi=\{a,
A\}$ as a set of two alleles for some genetic trait,  then
\[\Omega=\{\sigma_{1}=(a,a), \sigma_{2}=(a,A), \sigma_{3}=(A,a),
\sigma_{4}=(A,A)\}.\]
\[\Omega^{2}=\{\varphi_{ij}=(\sigma_{i},\sigma_{j}): i,j=1,2,3,4\},\ \
\Omega^{2}_{\varphi_{ii}}=\{\varphi_{ii}\},\]
\[\Omega^{2}_{\varphi_{ij}}=\{\varphi_{ii}, \varphi_{ij},
\varphi_{ji}, \varphi_{jj}\}, \ \ i,j=1,2,3,4.\] Consider a Gibbs
measure $\mu$ on $\Omega$: $\mu(\sigma_{i})=p_i>0$, with
$p_1+p_2+p_3+p_4=1$. Then we have
\[P_{\varphi_{ii}\psi}=\left\{\begin{array}{ll}
1 \ \ {\rm if} \ \ \psi=\varphi_{ii},\\
0 \ \ {\rm if} \ \ \psi\ne \varphi_{ii}
\end{array}\right.\ \ i=1,2,3,4.\] and
\[P_{\varphi_{ij}\psi}=\left\{\begin{array}{lll}
p_i^2(p_i+p_j)^{-2} \ \ {\rm if} \ \ \psi=\varphi_{ii},\\
p_j^2(p_i+p_j)^{-2} \ \ {\rm if} \ \ \psi=\varphi_{jj},\\
p_ip_j(p_i+p_j)^{-2} \ \ {\rm if} \ \ \psi=\varphi_{ij}, \psi_{ji}
\end{array}\right.\ \ i\ne j, \ i,j=1,2,3,4.\]
Correspondingly, the evolution algebra
$\mathcal{E}_{1}=\mathcal{E}(G,\Phi,\mu)$ is given by relations
\[\left\{\begin{array}{lll}
\varphi^{2}_{ii}=\varphi_{ii}, i=1,2,3,4\\
\varphi^{2}_{ij}=(p_i+p_j)^{-2}\left(p^2_i\varphi_{ii}+p_ip_j(\varphi_{ij}+\varphi_{ji})+p^2_j\varphi_{jj}\right), \ \ i\ne j,\\
\varphi\cdot\psi=0 \ \ \mbox{if} \ \ \varphi\ne \psi.\\
\end{array}\right.\ \ i,j=1,...,4. \eqno(8)\]
Denote the subalgebra generated by an element $\varphi$ by
$<\varphi>$. The algebra $\mathcal{E}_{1}$ has 6 subalgebras,
$<\varphi_{12}>$, $<\varphi_{13}>$, $<\varphi_{14}>$,
$<\varphi_{23}>$, $<\varphi_{24}>$, and $<\varphi_{34}>$. These
algebras are not simple, and each has dimension of 4. For example,
$<\varphi_{12}>=<\varphi_{11}, \varphi_{12}, \varphi_{21},
\varphi_{22}>$ has 4 generators, and its dimension also is 4. There
are four persistent generators, $<\varphi_{ii}>$, $i=1,2,3,4$. Each
of these persistent elements generates a simple evolution algebra
with dimension 1. According to \cite{tian}, algebras
$<\varphi_{11}>$, $<\varphi_{22}>$, $<\varphi_{33}>$ and
$<\varphi_{44}>$ are $0-th$ simple subalgebras. The 1st subalgebras
are generated by transit elements. There are 6 transit generators.
For example, $<\varphi_{12}>_{1}=<\varphi_{12}, \varphi_{21}>_{1}$.
Therefore, this algebra $\mathcal{E}$ has two levels in its
hierarchy, four simple subalgebras at the $0-th$, six simple
subalgebras at the $1-st$ level. Namely,
\[\mathcal{E}_{1}=<\varphi_{11}>\bigoplus <\varphi_{22}>\bigoplus
<\varphi_{33}>\bigoplus <\varphi_{44}>\dot{+} B_{0},\]
\[B_{0}=<\varphi_{12}>\bigoplus <\varphi_{13}>\bigoplus
<\varphi_{23}>\bigoplus <\varphi_{14}>\bigoplus
<\varphi_{24}>\bigoplus <\varphi_{34}>.\] There are two dynamical
flows from each subalgebra at the 1st level to two different
subalgebras at the 0-th level. For example, from the subalgebra
$<\varphi_{12}>$, one flow directs to the subalgebra
$<\varphi_{11}>$ and the other one directs to the subalgebra
$<\varphi_{22}>.$

\noindent {\bf Example 2:} Consider graph $G=(\Lambda, L)$ with
$\Lambda=\{1,2\}$, $ L=\emptyset$ and $\Phi=\{a,A\}.$  Then
$\Omega^{2}$, $\Omega^{2}_{\varphi_{ii}}$,
$\Omega^{2}_{\varphi_{ij}}$ are the same as the corresponding sets
in Example 1 if $(ij)\ne (14),(41),(23),(32)$. But
\[\Omega^{2}_{\varphi_{14}}=\Omega^{2}_{\varphi_{41}}=\Omega^{2}_{\varphi_{23}}=\Omega^{2}_{\varphi_{32}}=\Omega^{2},\]
which are large than corresponding sets in the Example 1.

Correspondingly, the evolution algebra
$\mathcal{E}_{2}=\mathcal{E}(G,\Phi,\mu)$ is given by relations
\[\left\{\begin{array}{llll}
\varphi^2_{ii}=\varphi_{ii}, i=1,2,3,4\\
\varphi^{2}_{ij}=(p_i+p_j)^{-2}\left(p^2_i\varphi_{ii}+p_ip_j(\varphi_{ij}+\varphi_{ji})+p^2_j\varphi_{jj}\right), i\ne j, (ij)\ne (14),(41),(23),(32)\\
\varphi^{2}_{23}=\varphi^{2}_{32}=\varphi^{2}_{14}=\varphi^{2}_{41}=\sum^4_{i,j=1}p_ip_j\varphi_{ij},\
\ (ij) =(14),(41),(23),(32)\\
\varphi\cdot\psi=0 \ \ \mbox{if} \ \ \varphi\ne \psi.\\
\end{array}\right. \]

The $\mathcal{E}_{2}$ has 3 levels in its hierarchy showed in the
following. \[\mathcal{E}_{2}=<\varphi_{11}>\bigoplus
<\varphi_{22}>\bigoplus <\varphi_{33}>\bigoplus
<\varphi_{44}>\dot{+} B_{0},\] \[B_{0}=<\varphi_{12}>\bigoplus
<\varphi_{13}>\bigoplus <\varphi_{24}>\bigoplus <\varphi_{34}>
\dot{+}B_{1},\] \[B_{1}=<\varphi_{14}>.\] Dynamically, there are 4
flows from the subalgebra $<\varphi_{14}>$ at the second level
direct to each subalgebra at the 1st level. There are two flows from
each subalgebra at the 1st level direct two different subalgebras at
the 0-th level.

\begin{remark}
If we assume $e_1=\varphi_{44}, e_2=\varphi_{11},
e_3=\varphi_{24}=\varphi_{42}=\varphi_{34}=\varphi_{43},
e_4=\varphi_{12}=\varphi_{21}=\varphi_{13}=\varphi_{31},
e_5=\varphi_{14}=\varphi_{41},
e_6=\varphi_{22}=\varphi_{33}=\varphi_{23}=\varphi_{32}$ and
$p_2=p_3$ then the evolution algebra is given by relations
\[\left\{\begin{array}{llll}
e_1^2=e_1,\ \ e_2^2=e_2,\ \ e_6^2=e_6\\
e_3^2=(p_2+p_4)^{-2}\left(p^2_4e_1+2p_2p_4e_3+p^2_2e_6\right),\\
e_4^2=(p_1+p_2)^{-2}\left(p^2_1e_2+2p_1p_2e_4+p^2_2e_6\right),\\
e_5^2=p^2_4e_1+p_1^2e_2+4p_2p_4e_3+4p_1p_2e_4+2p_1p_4e_5+4p^2_2e_6.\\
\end{array}\right. \]
This algebra is similar to the algebra considered in Example 7 of
(\cite{tian}, p.89), but does not coincide with it for any $p_i>0,
i=1,...,4$ i.e for any Gibbs measure $\mu$.
\end{remark}

\section{The case of infinite graphs}

We consider countable graphs to have countable many vertices and
edges. Let $G=(\Lambda,L)$ be a countable graph, and $\Phi$ be a
finite set. Then the set of all functions $\sigma:\Lambda\to\Phi$,
denoted by $\Omega$, is uncountable. Let $S(\Omega)$ be the set of
all probability measures defined on $(\Omega, {\cal F}),$ where
${\cal F}$ is the standard $\sigma-$algebra generated by the
finite-dimensional cylindrical set. Let $\mu$ be a measure on
$(\Omega, {\cal F})$ such that $\mu(B)>0$ for any finite-dimensional
cylindrical set $B\in {\cal F}.$ Note that only Gibbs measure have
this property \cite{preston}.

Fix a finite subset $M\subset \Lambda.$ We say that $\sigma\in
\Omega$ and $\varphi\in \Omega$ are equivalent if
$\sigma(M)=\varphi(M).$ Let $\xi=\{\Omega_{i}, i=1,2,...,q\},$
(where $q= |\Phi|^{|M|}$) be the partition of $\Omega$ generated by
this equivalent relation, $\Omega_{i}$ contains all equivalent
elements.

Let $G_M=(M, L_M)$ be finite subgraph of $G$ with $L_M=\{<x,y>\in L:
x,y\in M\}$. Let $M_i, i=1,...,m$ be maximal connected subgraphs of
$G_M$.

Consider $\Omega^{2}=\Omega\times\Omega$ as the set of all pairs of
parents. Fix $\sigma=(\varphi,\psi)\in\Omega^{2}$ and set

\[\Omega_{M,\sigma}=\left\{\omega\in \Omega: \omega(M_i)=\sigma(M_i),
\ \ \mbox{or}\ \ \omega(M_i)=\psi(M_i), \ \ i=1,...,m\right\}.\]

Define $P^M_{\varphi\psi}=P^M(\varphi,\psi)$ on
$\Omega^{2}_{M,\varphi}$ as
\[P^M(\varphi,\psi)=\left\{\begin{array}{ll}
\frac{\mu^2(\Omega^{2}_i)}{Z^M_{\varphi,i}}, \ \ {\rm if} \ \ \psi\in \Omega^{2}_{M,\varphi}\cap \Omega^{2}_i, \ i=1,...,q,\\
0 \ \ {\rm otherwise}
\end{array}\right.\ \ \eqno (9)\]
where \[Z^M_{\varphi,i}=\mu^2(\Omega^{2}_{M,\varphi}\cap
\Omega^{2}_{i})\sum^q_{j=1}\mu^2(\Omega^2_j).\]

Using (9) we now can define an evolution algebra $\mathcal{E}_M$ by
the following defining relations \[ \varphi\cdot
\varphi=\int_{\Omega^{2}}
P^M(\varphi,\psi)\psi=\sum^q_{i=1}{\mu^2(\Omega^{2}_i)\over
Z^M_{\varphi,i}} \int_{\Omega^{2}_{M,\varphi}\cap
\Omega^{2}_{i}}\psi, \ \ \varphi\in \Omega^{2}\]
\[\varphi\cdot\psi=0 \ \ \mbox{if} \ \ \varphi\ne \psi, \
\varphi,\psi\in \Omega^{2}\]

Note that by our construction the coefficients (9) depend on fixed
$M$. Consider an increasing sequence of connected, finite sets
$M_1\subset M_2\subset ...\subset M_n \subset ...$ such that $\cup_n
M_n=\Lambda$. For each $M_n$, using Gibbs measure $\mu_n$ defined on
$\Omega_{n}=\Phi^{M_n}$, we can define evolution algebras
$\mathcal{E}_{M_n}$, $n=1,2,...$

Denote by $P^{(n)}_{\varphi\psi}$ the coefficients (9) which is
constructed by $M_n$ and $\mu_n.$ An interesting problem is to
describe the set of all cells $\varphi, \psi$ such that the
following limits exist
$$P_{\varphi \psi}=\lim_{n\to\infty}P^{(n)}_{\varphi \psi}. \eqno(10)$$
Note that if $\varphi$ and $\psi$ are  equal almost sure (i.e. the
set $\{x: \varphi(x)\ne \psi(x)\}$ is finite) then the limit (10)
exists.

The evolution algebra $\mathcal{E}=\mathcal{E}(\mu)$ defined by
limit coefficients (10) and a limit Gibbs measure $\mu$ (see section
2) is called a limit evolution algebra.

These investigations allows to a natural introduction of
thermodynamics in studying of such evolution algebras. More
precisely, if $\Omega$ is continual set then one can associate the
Gibbs measure $\mu$ by a Hamiltonian $H$ (defined on $\Omega$) and
temperature $T>0$. It is known that depending on the Hamiltonian and
the values of the temperature the measure $\mu$ can be non unique.
In this case there is a phase transition of the physical system with
the Hamiltonian $H$. From the constructive definition of evolution
algebras, each $\mathcal{E}_{M_n}$ is a finite dimensional algebra,
but the limit evolution algebra will be an algebra with uncountable
dimension if it exists. Phase transitions in thermodynamics could
bring out some ideas in study continual algebras.

Let us consider the Potts model on $Z^{d}$ as an example of
Hamiltonian $H$. In this case $\Omega$ is the set of all cells
$\sigma:Z^{d}\to \Phi=\{1,...,q\}.$

The (formal) Hamiltonian of the Potts model is \[
H(\sigma)=-J\sum_{x,y\in Z^{d}\atop
\|x-y\|=1}\delta_{\sigma(x)\sigma(y)},\] where $J\in R.$

It is well-known (see \cite{sinai}, Theorem 2.3) that for the Potts
model with $q\geq 2$ there exist critical temperature $T_{\rm cr}$
such that for any $T<T_{\rm cr}$ there are $q$ distinct extreme
Gibbs measures $\mu^{(i)},$ $i=1,...,q.$

Let $M_n, n=1,2,...$ are fixed (as above) subsets of $Z^d$. Note
that, for $T$ low enough each measure $\mu^{(i)}$ is a small
deviation of the constant cells $\sigma^{(i)}\equiv i, i=1,...,q$.
This means that, when $T\to 0$, the restriction $\mu_n^{(i)}$ of
$\mu^{(i)}$ on $M_n$ converges (as $n\to \infty$) weakly towards
$\delta_{\sigma^{(i)}}$, the Dirac measure at the constant cell
$\sigma^{(i)}$. Thus $P^{(n)}_{\varphi\psi}\to 1,$  as $n\to\infty$
if $\varphi, \psi$ are equal to $\sigma^{(i)}$ almost sure and
$P^{(n)}_{\varphi\psi}\to 0$ otherwise. Hence, when $T\to 0$ then
Gibbs measures $\mu^{(i)}$, $i=1,..,q$ of the Potts model give $q$
different {\it limit} evolution algebras
$\mathcal{E}^{(i)}=\langle\sigma^{(i)}\rangle$, $i=1,...,q$.

\noindent {\bf Open problems:} We have defined evolution algebras
$\mathcal{E}=\mathcal{E}(G,\Phi,\mu)$ for a given graph, state
spaces and Gibbs measures. If graphs are finite and connected, we
obtained a structure theorem for this type of algebras. If graphs
are finite but not connected, we don't know their precise
structures. If graphs are countable, we do not know any deep
results. We therefore list several interesting open problems here.

1. If graphs are finite and not connected, what are structures of
$\mathcal{E}(G,\Phi,\mu)$? For example, if a graph has two
components, how many simple algebras does $\mathcal{E}(G,\Phi,\mu)$
have? How does its hierarchy look like?

2. If graphs are countable and connected, what are the structures of
algebras $\mathcal{E}(G,\Phi,\mu)$? Do we have a similar theorem as
that in the case of finite graphs \ref{structure}?

3. For a fixed countable graph, how do Gibbs measures affect the
limit evolution algebras defined in this section? From the example
on Potts model above we know there could be several or many
different limit evolution algebras. Are these limit algebras
isomorphic or homomorphic? Do we have a similar theorem as the
Theorem \ref{structure-0}?

4. For countable graphs, is there any condition under which the
limit evolution algebra exist?

5. More generally, for countable graphs how does the thermodynamics
(the phase transition) affect properties of evolution algebras
corresponding to Gibbs measures of the Hamiltonian $H$? How do
evolution algebras reflect thermodynamics?

\begin{acknowledgement}
U. A. Rozikov thanks the Abdus Salam International Center for
Theoretical Physics (ICTP), Trieste, Italy for providing financial
support of his visit to ICTP (February-April 2009). J. P. Tian is
supported by New faculty start-up fund in the College of William and
Mary, USA.
\end{acknowledgement}

\end{document}